\documentclass[11pt]{article}
\usepackage[margin=2.5cm]{geometry}
\usepackage{amsmath, amsthm, amsfonts, amssymb}
\usepackage[mathscr]{euscript}
\usepackage{amsmath}
\usepackage{array, float}

\usepackage{graphicx, latexsym}
\usepackage{lscape}
\usepackage{multicol}
\usepackage{multirow}
\usepackage{graphicx}
\usepackage{subfig}
\usepackage{tikz}
\usepackage{cite}
\newtheorem{thm}{Theorem}

\newtheorem{cor}{Corollary}
\newtheorem{obs}{Observation}
\newtheorem{lem}[thm]{Lemma}

\newtheorem{discu}{Discussion:}

\newtheorem{conje}{Conjecture:}

\usepackage{fancyhdr}
\begin{document}
	\begin{center}
		\textbf{\huge The Edge Geodesic Transversal Problem}\\
	\end{center}
	\begin{center}
		\textbf{R. Prabha$^{\text a}$, B. Sandhiya$^{\text a,*}$, Paul Manuel$^{\text b}$}
	\end{center}
	\begin{center}
		\small $^{\text a}$Department of Mathematics, Ethiraj College for Women, Chennai 600 008, India\\
		\small $^{\text b}$Department of Information Science, College of Computing Science and Engineering, Kuwait University, Kuwait\\
	\end{center}
	\begin{center}
		*Corresponding author: sandhiyabasskar14@gmail.com
	\end{center}
	\begin{abstract}
For a given connected graph $G(V,E)$, the geodesic transversal problem finds a set $S$ of vertices of $G$ with minimum cardinality such that every maximal geodesic of $G$ contains at least one vertex of $S$. The cardinality of such a smallest set $S$ is called the gt-number of $G$ and is denoted as $gt(G)$. In this paper, we introduce the \textit{Edge geodesic transversal problem} as the edge-based extension of the geodesic transversal problem. The problem of determining a set $S\subseteq E(G)$ with minimum number of elements such that every maximal geodesic $P$ in $G$ traverses at least one edge of $S$ is called the edge geodesic transversal problem, and the cardinality of such a set is called the edge geodesic transversal number or the $gt_e$-number of $G$, denoted as $gt_e(G)$. Further, we investigate the effect of an edge cut of a graph on the edge geodesic transversal problem and prove a novel result called the edge cut transversal theorem which is used in the computation of edge geodesic transversal number of hypercube, stacked prism and ladder graph. 

\end{abstract}
\hspace*{0.5cm}\textbf
 
\section{Introduction}

Graphs are widely used to represent real-world systems in which objects(nodes or vertices) interact through connections(edges). Examples include communication networks, transportation systems, social networks, chemical structures, and biological structures. A fundamental objective in the study of graphs is to analyze how information, resources, or entities move through such networks. For this reason, shortest paths play a crucial role in graph theory.

Given a connected undirected graph $G(V,E)$, the shortest path between two vertices $u,v\in V(G)$ is called a geodesic. Geodesics provide a natural measure of efficiency in a network because they represent routes or paths with minimum distance. As a result, geodesic-based investigations have become an important area of research in both theoretical and applied graph theory. Geodesic-based concepts play a significant role in optimization problems that aim to identify a vertex or a set of vertices with a distinguished role in relation to shortest paths. Such investigations help identify critical locations within a network and have applications in routing, network monitoring, facility location, communication networks, and network control.

The geodesic transversal problem is one such optimization problem which was introduced by Bre\v{s}ar, Klav\v{z}ar and Manuel \cite{1st paper} in 2022. A set $S\subseteq V(G)$ is a geodesic transversal of $G$, if the vertices of $S$ collectively hits all maximal geodesics in $G$. A geodesic transversal with minimum elements is called a geodesic transversal set or $gt$-set and its cardinality is the geodesic transversal number or the $gt$-number of $G$.

While the geodesic transversal problem focuses on selecting vertices that intersect every maximal geodesic, a natural extension of this concept considers the role of edges. In many network applications, communication and transportation occur through links rather than through individual nodes$\backslash$vertices. This observation motivates the introduction of the edge geodesic-transversal problem. 

An edge geodesic-transversal is a set $S\subseteq E(G)$ such that every maximal geodesic of $G$ passes through at least one edge of $S$. An edge geodesic-transversal with minimum cardinality is called the edge geodesic transversal set or the $gt_e$-set and the number of edges in $S$ is called the edge geodesic transversal number or the $gt_e$-number of $G$, denoted as $gt_e(G)$. The problem of finding the $gt_e$-number of any given connected graph is known as the edge geodesic transversal problem. 

In this paper, we have solved the edge geodesic transversal problem for $r$-dimensional hypercube. The $r$-dimensional hypercube is a fundamental interconnection network that has been widely investigated in graph theory and parallel computing. Its regular structure, high degree of symmetry, strong connectivity, and logarithmic diameter enable efficient communication and routing among networks. For basic properties of the hypercube, the reader is referred to  \cite{hyper1}, \cite{hyper2} and \cite{hyper3}.

\section{Overview of the Paper}

We begin the study of edge geodesic transversal problem in section $3$, by establishing two fundamental bounds. First, we show that $gt_e(G)$ is bounded below by the $gt(G)$. Then we prove that $gt_e(G)$ is bounded above by the sum of the degrees of the vertices belonging to a $gt$-set of $G$ and further we compute the $gt_e$-number of some basic graphs. In section 4, we prove a theorem namely, the \textit{edge cut transversal theorem}, which is the most significant result of this paper. Finally, using the edge cut transversal theorem we determine the $gt_e$-number of hypercube, stacked prism and ladder graph.

\section{Preliminaries}

Consider a path $P_n$ on $n$ vertices. Clearly, the path is the only maximal geodesic and every single edge is an edge geodesic transversal set and hence $gt_e(P_n)=1$. In case of a complete graph $K_n, n\geq1$, every edge is a maximal geodesic and hence $gt_e (K_n)=n(n-1)/2$. We begin with the following lemma which establishes a relationship between $gt_e(G)$ and $gt(G)$. 
\begin{lem}
   If $G$ is a connected graph, then $gt_e(G)\geq gt(G)$. 
\end{lem}

\begin{proof}
    Let $T$ be a $gt_e$-set of $G$ with cardinality $k$. Suppose that $T=\{e_i:u_iv_i|i=1,2,...,k\}$. Then, the collection of one of the terminal vertices of these $k$ edges, say $S_1 =\{u_1, u_2, \ldots, u_k\}$ clearly forms a geodesic transversal of $G$. Thus, $gt(G)\leq k=gt_e(G)$. 
\end{proof}

\begin{lem}
    Let $G(V,E)$ be a connected graph and $S=\{u_1, u_2 ,\ldots u_k\}$ be a gt-set of $G$. Then, $gt_e(G)\leq \sum_{i=1}^kd(u_i), u_i\in S$. 
\end{lem}

\begin{proof}
    Every maximal geodesic of $G$ meets at least one $u_i\in S$, through an edge incident to it. Then, the set $T$ of edges incident on every $u_i, i\in [1,k]$ will be an edge geodesic transversal of $G$.  Hence, $gt_e(G)\leq |T|=\sum_{i=1}^kd(u_i), u_i\in S$. 
\end{proof}

\begin{obs}
    The following results hold for standard graph families:
    \begin{itemize}
        \item $gt_e(P_n)=1$
        \item $gt_e(C_n)=\begin{cases}
            2, n \text{ is even.}\\
            3, n \text{ is odd.}            
        \end{cases}$
        \item $gt_e(K_n)=n(n-1)/2$
        \item $gt(K_{1,n})=n-1$
    \end{itemize}

\end{obs}

\section{Edge cut Transversal Theorem}
Let $G=(V,E)$ be a connected graph and $X\subseteq V(G)$. An edge cut is a subset of edges of the form $E[X,V\backslash X]$, consisting of all edges with one end vertex in $X$ and the other end in the complement $V\backslash X$. The removal of an edge cut from $G$ disconnects $G$. In this section, we derive a relationship between the cardinality of an edge cut and the edge geodesic transversal number of a graph $G$, subject to certain conditions.

\begin{lem}\label{lem3}
    Let $G$ be a connected graph and $S$ be an edge cut of $G$. If for every maximal geodesic $P$ in $G$, $|E(P)\cap S|=1$, then $gt_e(G)\leq |S|$.
\end{lem}
\begin{proof}
    Assume the contrary that $gt_e(G)> |S|$ and let $S'$ be a $gt_e$-set of $G$ such that $|S'|>|S|$. Clearly, $S'$ contains at least one edge $e\in S'\backslash S$. If every maximal geodesic through $e$ intersects $S$, then $S'\backslash \{e\}$ remains a $gt_e$-set which contradicts the minimality of $S'$. Hence, there exists a maximal geodesic $P_e$ in $G$ containing $e$ such that $|E(P_e)\cap S|=\emptyset$. But then, $|E(P)\cap S|=\emptyset$ contradicts the hypothesis of the lemma and hence $gt_e(G)\leq |S|$. 
\end{proof}

\begin{lem}\label{lem4}
    If in a connected graph $G$, there exist $k$ mutually edge disjoint maximal geodesics, then $gt_e(G)\geq k$.
\end{lem}
\begin{proof}
    Every edge geodesic transversal of $G$ intersects each of $k$ mutually edge disjoint maximal geodesics of $G$, hence $gt_e(G)\geq k$.
\end{proof}

From Lemma \ref{lem3} and \ref{lem4}, we derive the following edge cut transversal theorem which states that if an edge cut $S$ separates $G$ into two connected components  such that every maximal geodesic $P$ of $G$ contains exactly one edge of $S$ and $G$ admits at least $|S|$ mutually edge disjoint maximal geodesics, then the edge geodesic transversal number of $G$ is equal to the cardinality of $G$.
\begin{thm}\label{edgecutthm}
    \textbf{Edge-cut transversal Theorem} \\ Let $S=\{e_1,e_2,...,e_k\}$ be an edge cut of $G$ satisfying the following conditions.
   \begin{enumerate}
       \item The removal of $S$ disconnects $G$ into two connected components $G_1$ and $G_2$.
       \item $|E(P)\cap S|=1$, for every maximal geodesic $P$ in $G$.
       \item $G$ has at least $k$ mutually edge disjoint maximal geodesics.
   \end{enumerate}
   Then $gt_e(G)=|S|$.
\end{thm}

We now recall the definition of hypercube, stacked prism and ladder graph. For $r\geq 1$, $Q_r$ denotes $r$-dimensional hypercube. The vertex set of $Q_r$ is formed by the collection of all $r$-dimensional binary representations. Two vertices $x,y\in V(Q_r)$ are adjacent iff the corresponding binary representations differ by exactly one bit. The length of a diametral path of $Q_r$ is $r$ \cite{hypercube}. A stacked prism is defined as $Y_{m,n}=C_m\square P_n$ for $m\geq 3$ and $n>1$. A stacked prism is also known as generalized prism graph \cite{stackedprism}. The ladder graph, denoted as $L_n$ is the cartesian product $P_2\square P_n$ \cite{ladder}. 

\begin{cor}\label{hyperresult}
    If $Q_r$ is a hypercube of $r$- dimension, then $gt(Q_r)=2^{r-1}$.
\end{cor}
\begin{proof}
    $Q_r$ is obtained from two copies of $Q_{r-1}$, namely $Q_{r-1}^0$ and $Q_{r-1}^1$ . If $x\in V(Q_{r-1})$, then the corresponding vertices in the two copies of $Q_r$ are denoted by $0x$ and $1x$ respectively. Hence, $V(Q_n)=\{0x:x\in V(Q_{r-1})\}\cup \{1x:x\in V(Q_{r-1})\}$. We observe that the set $S=\{(0x,1x): 0x\in V(Q_{r-1}^0), 1x\in V(Q_{r-1}^1)\}$ of edges that joins the two copies is an edge cut of $Q_r$.
    We now claim that every maximal geodesic of $Q_r$ contains exactly one edge of $S$. Suppose on the contrary, there exist a maximal geodesic $P(x,y)$ in $Q_r$ containing two edges $e_1,e_2\in S$, then both end vertices of $P$ lie either in $Q_{r-1}^0$ or $Q_{r-1}^1$. Without loss of generality assume $x,y\in Q_{r-1}^0$. As every maximal geodesic is a diametral path, $l(P)=r$. But then there exists a geodesic $P^0(x,y)$ that lies entirely in $Q_{r-1}^0$ such that $l(P^0)\leq r-1$ since $diam(Q_{r-1}^0)=r-1$. This contradicts the assumption that $P$ is a maximal geodesic. Thus, for every maximal geodesic $P$ in $Q_r$, $|E(P)\cap S|=1$. 

    \begin{figure}[htbp]
\centering
\begin{tikzpicture}[
    vertex/.style={
        circle,
        draw=black,
        fill=white,
        line width=1pt,
        minimum size=2.5mm,
        inner sep=0pt
    },
    edge/.style={black, line width=1pt},
    rededge/.style={
        red!75!black,
        dashed,
        line width=1pt,
        dash pattern=on 5pt off 4pt
    }
]

% Outer square vertices
\node[vertex] (a) at (0,5.2) {};
\node[vertex] (b) at (5.2,5.2) {};
\node[vertex] (c) at (0,0) {};
\node[vertex] (d) at (5.2,0) {};

% Inner square vertices
\node[vertex] (e) at (1.2,4.0) {};
\node[vertex] (f) at (4.0,4.0) {};
\node[vertex] (g) at (1.2,1.2) {};
\node[vertex] (h) at (4.0,1.2) {};

% Red dashed edges (drawn first)
\draw[rededge] (a) -- (e);
\draw[rededge] (b) -- (f);
\draw[rededge] (c) -- (g);
\draw[rededge] (d) -- (h);

% Outer square edges
\draw[edge] (a) -- (b);
\draw[edge] (b) -- (d);
\draw[edge] (d) -- (c);
\draw[edge] (c) -- (a);

% Inner square edges
\draw[edge] (e) -- (f);
\draw[edge] (f) -- (h);
\draw[edge] (h) -- (g);
\draw[edge] (g) -- (e);

% Labels: positioned outside the vertices
\node[anchor=east, xshift=-3pt] at (a) {$010$};
\node[anchor=west, xshift=3pt] at (b) {$011$};
\node[anchor=east, xshift=-3pt] at (c) {$000$};
\node[anchor=west, xshift=3pt] at (d) {$001$};

\node[anchor=east, xshift=-3pt] at (e) {$110$};
\node[anchor=west, xshift=3pt] at (f) {$111$};
\node[anchor=east, xshift=-3pt] at (g) {$100$};
\node[anchor=west, xshift=3pt] at (h) {$101$};

\end{tikzpicture}
\caption{An edge cut of $Q_3$.}
\label{edgecutQ2Q3}
\end{figure}
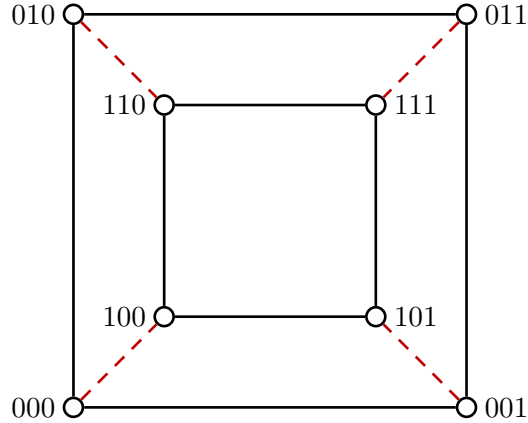

\begin{figure}[htbp]
\centering
\begin{tikzpicture}[
    vertex/.style={
        circle,
        draw=black,
        fill=white,
        line width=1pt,
        minimum size=2.5mm,
        inner sep=0pt
    },
    edge/.style={line width=1pt}
]

% Outer square vertices
\node[vertex] (a) at (0,5) {};
\node[vertex] (b) at (5,5) {};
\node[vertex] (c) at (0,0) {};
\node[vertex] (d) at (5,0) {};

% Inner square vertices
\node[vertex] (e) at (1.25,3.75) {};
\node[vertex] (f) at (3.75,3.75) {};
\node[vertex] (g) at (1.25,1.25) {};
\node[vertex] (h) at (3.75,1.25) {};

% Outer edges
\draw[edge, red!75!black] (a) -- (b);
\draw[edge, red!75!black] (a) -- (c);
\draw[edge, blue!75!black] (b) -- (d);
\draw[edge, blue!75!black] (c) -- (d);

% Inner edges
\draw[edge, green!70!black] (e) -- (f);
\draw[edge, green!70!black] (f) -- (h);
\draw[edge, magenta] (e) -- (g);
\draw[edge, magenta] (g) -- (h);

% Connecting edges
\draw[edge, green!70!black] (a) -- (e);
\draw[edge, red!75!black] (f) -- (b);
\draw[edge, blue!75!black] (c) -- (g);
\draw[edge, magenta] (h) -- (d);

% Vertex labels
\node[anchor=east, xshift=-3pt] at (a) {$010$};
\node[anchor=west, xshift=3pt] at (b) {$011$};
\node[anchor=east, xshift=-3pt] at (c) {$000$};
\node[anchor=west, xshift=3pt] at (d) {$001$};

\node[anchor=east, xshift=-3pt] at (e) {$110$};
\node[anchor=west, xshift=3pt] at (f) {$111$};
\node[anchor=east, xshift=-3pt] at (g) {$100$};
\node[anchor=west, xshift=3pt] at (h) {$101$};

\end{tikzpicture}
\caption{Four edge disjoint maximal geodesics in $Q_3$.}
\label{disjointmaxgeodesics}
\end{figure}
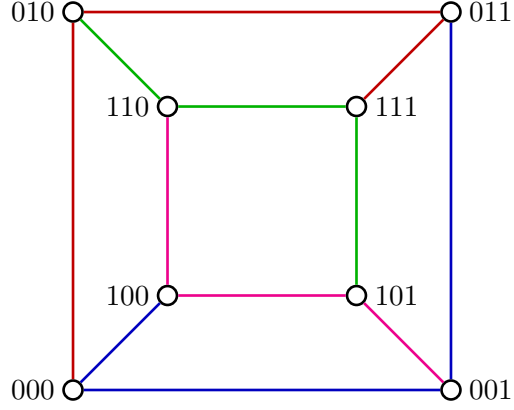

    Our second claim is that  $Q_r$ contains $2^{r-1}$ edge disjoint maximal geodesics. We prove this by induction on $r$. For the base case $r=3$, Figure \ref{disjointmaxgeodesics} illustrates $2^{3-1}=4$ edge disjoint maximal geodesics in $Q_3$. Assume that the statement holds for $r=k-1$, that is $Q_{k-1}$ contains $2^{k-2}$ edge disjoint maximal geodesics and let $\mathcal{P}=\{P_1, P_2,...,P_{2^{k-2}}\}$ where each $P_j=P_j(x_i,\Bar{x_i}), x_i\in V(Q_{k-1}), 1\leq j\leq 2^{k-2}$ be the set of edge disjoint maximal geodesics in $Q_{k-1}$. We now construct $2^{k-1}$ edge disjoint maximal geodesics in $Q_k$. For each $P_j(x_i,\Bar {x_i})$, construct the path $P_j':(0x_i)...(0\Bar {x_i})(1,\Bar {x_i})$ in $Q_r$ containing $P_j(0x_i,0\Bar{x_i})$ as its sub-path. Analogously, we construct $P_j'':(1 \Bar{x_i})...(1{x_i})(0,{x_i})$ in $Q_r$ containing $P_j(1x_i,1\Bar{x_i})$ as its sub-path, refer Figure \ref{q4}.  Since the geodesics in $\mathcal{P}$ are pair-wise edge disjoint, the constructed paths are also mutually edge disjoint. Hence, the set $\mathcal{T}=\{P_j'\}\cup \{P_j''\}, 1\leq j\leq 2^{k-2}$ is a set of mutually edge disjoint maximal geodesics of $Q_k$ with cardinality $2^{k-1}$. Thus, $Q_r$ satisfies the hypotheses of edge-cut transversal theorem and application of the theorem completes the proof.
\end{proof}

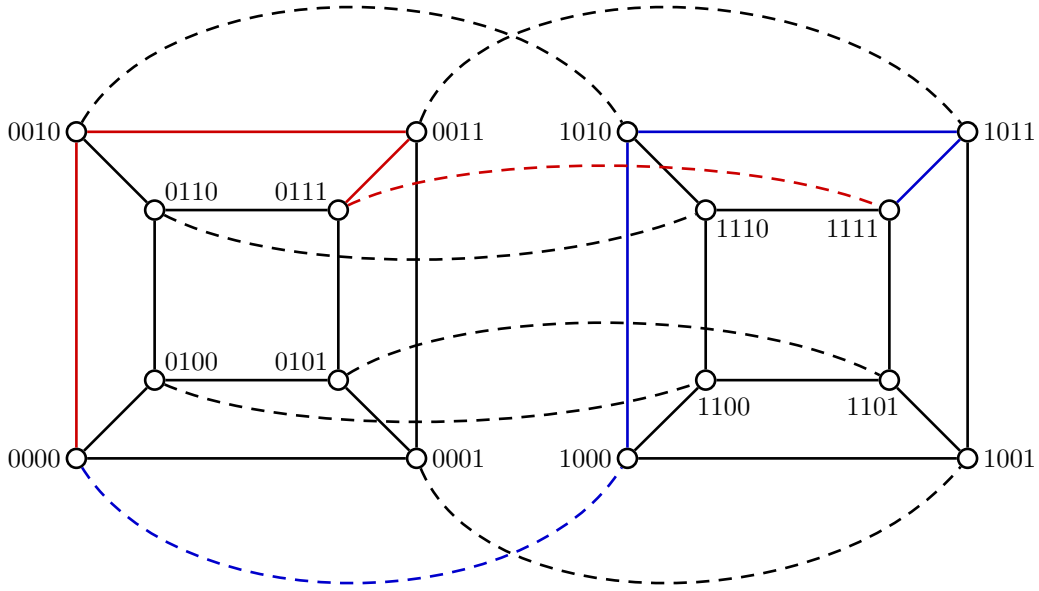
\begin{figure}[htbp]
\centering
\begin{tikzpicture}[
    x=0.9cm,
    y=0.9cm,
    vertex/.style={
        circle,
        draw=black,
        fill=white,
        line width=1pt,
        minimum size=2.5mm,
        inner sep=0pt
    },
    edge/.style={draw=black, line width=1pt},
    rededge/.style={draw=red!80!black, line width=1pt},
    blueedge/.style={draw=blue!80!black, line width=1pt},
    dashededge/.style={
        dashed,
        line width=1pt,
        dash pattern=on 5pt off 3.5pt
    },
    lab/.style={font=\small}
]

% Left cube vertices
\node[vertex] (v0000) at (0,0) {};
\node[vertex] (v0001) at (5,0) {};
\node[vertex] (v0010) at (0,4.8) {};
\node[vertex] (v0011) at (5,4.8) {};

\node[vertex] (v0100) at (1.15,1.15) {};
\node[vertex] (v0101) at (3.85,1.15) {};
\node[vertex] (v0110) at (1.15,3.65) {};
\node[vertex] (v0111) at (3.85,3.65) {};

% Right cube vertices
\node[vertex] (v1000) at (8.1,0) {};
\node[vertex] (v1001) at (13.1,0) {};
\node[vertex] (v1010) at (8.1,4.8) {};
\node[vertex] (v1011) at (13.1,4.8) {};

\node[vertex] (v1100) at (9.25,1.15) {};
\node[vertex] (v1101) at (11.95,1.15) {};
\node[vertex] (v1110) at (9.25,3.65) {};
\node[vertex] (v1111) at (11.95,3.65) {};

% Left cube outer edges
\draw[rededge] (v0000)--(v0010);
\draw[rededge] (v0010)--(v0011);
\draw[edge] (v0011)--(v0001);
\draw[edge] (v0001)--(v0000);

% Left cube inner edges
\draw[edge] (v0100)--(v0101);
\draw[edge] (v0101)--(v0111);
\draw[edge] (v0111)--(v0110);
\draw[edge] (v0110)--(v0100);

% Left cube connecting edges
\draw[edge] (v0000)--(v0100);
\draw[edge] (v0001)--(v0101);
\draw[edge] (v0010)--(v0110);
\draw[rededge] (v0011)--(v0111);

% Right cube outer edges
\draw[blueedge] (v1000)--(v1010);
\draw[blueedge] (v1010)--(v1011);
\draw[edge] (v1011)--(v1001);
\draw[edge] (v1001)--(v1000);

% Right cube inner edges
\draw[edge] (v1100)--(v1101);
\draw[edge] (v1101)--(v1111);
\draw[edge] (v1111)--(v1110);
\draw[edge] (v1110)--(v1100);

% Right cube connecting edges
\draw[edge] (v1000)--(v1100);
\draw[edge] (v1001)--(v1101);
\draw[edge] (v1010)--(v1110);
\draw[blueedge] (v1011)--(v1111);

% Dashed curved edges
\draw[edge,dashededge]
    (v0010) .. controls (1.4,7.2) and (6.7,7.2) .. (v1010);

\draw[edge,dashededge]
    (v0011) .. controls (5.9,7.2) and (11.2,7.2) .. (v1011);

\draw[edge,dashededge]
    (v0110) .. controls (3.0,2.7) and (6.9,2.7) .. (v1110);

\draw[rededge,dashededge]
    (v0111) .. controls (5.6,4.50) and (9.8,4.50) .. (v1111);

\draw[edge,dashededge]
    (v0100) .. controls (3.0,0.35) and (6.9,0.35) .. (v1100);

\draw[edge,dashededge]
    (v0101) .. controls (5.6,2.25) and (9.8,2.25) .. (v1101);

\draw[blueedge,dashededge]
    (v0000) .. controls (1.5,-2.4) and (6.6,-2.4) .. (v1000);

\draw[edge,dashededge]
    (v0001) .. controls (6.0,-2.4) and (11.1,-2.4) .. (v1001);

% Vertex labels: left cube outer vertices
\node[lab,left=2pt]  at (v0010) {$0010$};
\node[lab,right=2pt] at (v0011) {$0011$};
\node[lab,left=2pt]  at (v0000) {$0000$};
\node[lab,right=2pt] at (v0001) {$0001$};

% Vertex labels: left cube inner vertices
\node[lab, above right]  at (v0110) {$0110$};
\node[lab,above left]  at (v0111) {$0111$};
\node[lab,above right] at (v0100) {$0100$};
\node[lab,above left]  at (v0101) {$0101$};

% Vertex labels: right cube outer vertices
\node[lab,left=2pt]  at (v1010) {$1010$};
\node[lab,right=2pt] at (v1011) {$1011$};
\node[lab,left=2pt]  at (v1000) {$1000$};
\node[lab,right=2pt] at (v1001) {$1001$};

% Vertex labels: right cube inner vertices
\node[lab,below right] at (v1110) {$1110$};
\node[lab,below left]  at (v1111) {$1111$};
\node[lab, right=-7pt, yshift=-10pt] at (v1100) {$1100$};
\node[lab,right=-20pt, yshift=-10pt]    at (v1101) {$1101$};

\end{tikzpicture}
\caption{$P'(0000,1111)$ containing $P(0000,0111)$ of $Q_3^0$ and $P''(1111,0000)$ containing $P(1111,1000)$ of $Q_3^1$ shown in red and blue lines respectively.}
\label{q4}
\end{figure}

\begin{cor}
    If $L_n$ is a ladder graph, then $gt(L_n)=2$.
\end{cor}
\begin{proof}
    The ladder graph can be considered as the cartesian product $P_n\square K_2$ and the proof is identical to that of Corollary \ref{hyperresult}. 
\end{proof}

\begin{cor}
    If $SP_{n,k}$ is a stacked prism, then $gt(SP_{n,k})=k$.
\end{cor}
\begin{proof}
    The stacked prism $SP_{n,k}$ is the cartesian product $P_n\square C_k$ and the proof is identical to that of Corollary \ref{hyperresult}. 
\end{proof}
\section{Conclusion}

In this paper, we introduce the edge geodesic transversal problem, the edge version of the geodesic transversal problem and compute an upper and lower bound for the $gt_e$-number of $G$. We also prove a general result, the edge cut transversal theorem, which has a lot of significance in the computation of edge geodesic transversal number of several graphs. These results not only broaden the theoretical understanding of geodesic covering parameters but also provide further scope for investigating the edge geodesic transversal problem in other graph products, such as the lexicographic and corona products of graphs.

\end{document}